\documentclass{amsart}
\usepackage{bbm,mathtools,enumitem,amssymb}
\usepackage{amsrefs}
\allowdisplaybreaks

\newcommand{\R}{\ensuremath{\mathbb{R}}}
\newcommand{\Exp}{\ensuremath{\mathbb{E}}}
\newcommand{\F}{\ensuremath{\mathcal{F}}}

\newcommand{\Sc}{\ensuremath{\mathcal{S}}}

\newcommand{\Ltwo}{\ensuremath{\mathcal{L}^2}}

\newcommand{\inpr}[3][]{\left\langle#2 \,,\, #3\right\rangle_{#1}}
\newcommand{\predbracket}[1]{\left\langle#1\right\rangle}
\newcommand{\bracket}[1]{\left[#1\right]}
\newcommand{\indicator}[1]{\mathbbm{1}_{#1}}
\newcommand{\Partition}{\ensuremath{\mathbb{P}}}

\DeclarePairedDelimiter\floor{\lfloor}{\rfloor}
\theoremstyle{plain}
\newtheorem{thm}{Theorem}[section]

\newtheorem{lem}[thm]{Lemma}
\newtheorem{propn}[thm]{Proposition}
\theoremstyle{definition}
\newtheorem{rem}[thm]{Remark}

\newtheorem{note}[thm]{Note}

\numberwithin{equation}{section}

\begin{document}
\title[An It\={o} formula in $\mathcal S'$]{An It\={o} formula in the space of 
tempered distributions}
\author{Suprio Bhar}
\address{Suprio Bhar, Indian Statistical Institute Bangalore Centre.}
\date{}
\email{suprio@isibang.ac.in}
\begin{abstract}
We extend the It\={o} formula \cite{MR1837298}*{Theorem 2.3} for 
semimartingales 
with rcll paths. We 
also comment on Local time process of such semimartingales. We apply the 
It\={o} formula to L\'{e}vy processes to obtain existence of solutions to 
certain classes of stochastic differential equations in the 
Hermite-Sobolev spaces.
\end{abstract}
\maketitle
\section{Introduction}
It\={o} formula is an important result in stochastic calculus and has been 
studied in quite generality, starting from real valued processes to processes 
taking values in Nuclear spaces (\cites{MR1467435, MR3063763, 
MR688144, MR0264754, MR578177, MR664333, MR1837298, 
MR1465436, MR771478, MR1207136}).

Let $\Sc(\R^d)$ denote the space of real valued rapidly decreasing smooth 
functions on $\R^d$ and let 
$\Sc'(\R^d)$ denote the dual space, i.e. the space of tempered distributions. 
For $p \in \R$, let $\Sc_p(\R^d)$ denote the Hermite-Sobolev spaces and for 
$x\in \R^d$, let $\tau_x$ denote the translation operators (see definitions 
in Section 2). Given $\phi \in \Sc_{-p}(\R^d)$ and an $\R^d$ valued continuous 
semimartingale $X_t=(X_t^1,\cdots,X_t^d)$, we have the following It\={o} 
formula (see \cite{MR1837298}*{Theorem 2.3})
\begin{thm}
$\{\tau_{X_t}\phi\}$ is an $\Sc_{-p}(\R^d)$ valued continuous semimartingale 
and we have the equality in $\Sc_{-p-1}(\R^d)$, a.s.
\[\tau_{X_t}\phi = \tau_{X_0}\phi - \sum_{i=1}^d \int_0^t
\partial_i\tau_{X_{s}}\phi\,
dX^i_s + \frac{1}{2}\sum_{i,j=1}^d \int_0^t \partial_{ij}^2\tau_{X_{s}}\phi\,
d[X^i,X^j]_s,\, t \geq 0.\]
\end{thm}
This result has been used in \cite{MR3063763} to show
existence of solution of some stochastic differential equations in 
$\Sc'(\R^d)$. The aim of the current paper is to prove the result 
for semimartingales $\{X_t\}$
with rcll (right continuous with left limits) paths.

A version of this It\={o} formula was also 
proved in \cite{MR664333}*{Theorem III.1} with equality in $\Sc'$. In 
\cite{MR0264754}*{Theorem 3}, the 
author has proved this formula for twice continuously (Fr\'{e}chet) 
differentiable function while dealing with a single 
Hilbert space. Note that derivatives of tempered distributions may not be 
in the same Hermite Sobolev space as the original one. Using 
regularization, the result \cite{MR0264754}*{Theorem 3} 
was also proved in \cite{MR613311}*{Theorem 8} in the case of an $E'$ valued 
continuous martingale, where $E$ is a countably Hilbertian Nuclear space.

In Section 2, we recall the countably Hilbertian topology defined on
$\Sc(\R^d)$ which gives rise to the Hermite-Sobolev spaces $\Sc_p(\R^d)$.

In Section 3, we provide the construction of the stochastic integral of an 
$\Sc_{-p}(\R^d)$ valued norm-bounded predictable process $\{G_t\}$ with 
respect to a real valued semimartingale $\{X_t\}$ from the first principles. 
Since the semimartingale is real valued, this procedure is simpler than the 
Hilbert valued stochastic integration described in 
\cite{MR688144}*{Chapter 4 
and 
5}. We note that for any 
$\phi 
\in \Sc(\R^d)$, a.s.
\[\inpr{\int_0^t G_s\,dX_s}{\phi}=\int_0^t \inpr{G_s}{\phi}\, dX_s,\, t \geq 
0.\]
We exploit this property to prove an It\={o} formula (see 
Theorem \ref{Ito-formula}). As an application, we consider an one-dimensional 
L\'{e}vy process $X$ and show the existence of a solution of a stochastic 
differential equation (Theorem \ref{Levy}) in the Hermite-Sobolev spaces. This 
is 
similar to the solution obtained 
in \cite{MR3063763} for continuous processes $X$.

\section{Topologies on $\Sc$ and $\Sc'$}
Let $\Sc(\R^d)$ be the space of smooth rapidly 
decreasing
$\R$-valued functions on $\R^d$ with the topology given by L. Schwartz (see
\cite{MR2296978}) and let $\Sc'(\R^d)$ be the dual space, known as the 
space of tempered distributions. Let $\Sc_p(\R^d)$ be the completion of 
$(\Sc(\R^d), \|\cdot\|_p)$ for any $p \in \R$ (see
\cite{MR771478}*{Chapter 1.3} for the notations). The spaces $\Sc_p(\R^d), p 
\in 
\R$ are separable Hilbert spaces and are known as the Hermite-Sobolev spaces. 
We write $\Sc, \Sc', \Sc_p$ instead of $\Sc(\R), \Sc'(\R), \Sc_p(\R)$.\\
Note that $\Sc_0(\R^d) = \Ltwo(\R^d)$ and for 
$p>0$, $(\Sc_{-p}(\R^d), \|\cdot\|_{-p})$ is dual to $(\Sc_p(\R^d),
\|\cdot\|_p)$. Furthermore,
\[\Sc(\R^d) = \bigcap_{p \in \R}(\Sc_p(\R^d), \|\cdot\|_p), \quad 
\Sc'(\R^d) = \bigcup_{p \in \R}(\Sc_p(\R^d), \|\cdot\|_p)
\]
Given $\psi \in \Sc(\R^d)$ (or $\Sc_p(\R^d)$) and $\phi \in \Sc'(\R^d)$
(or
$\Sc_{-p}(\R^d)$), the action of $\phi$ on $\psi$ will be denoted by
$\inpr{\phi}{\psi}$.\\
Let $\{h_n: n \in {\mathbb Z}_+^d\}$ be the Hermite functions, where ${\mathbb 
Z}_+^d := \{n= (n_1,\cdots, n_d) : n_i\, 
\text{non-negative integers}\}$.
If $n = (n_1, \cdots ,n_d)$, we define $|n|
:= n_1+\cdots +n_d$. Note that $\{h_n^p: n \in {\mathbb Z}_+^d\}$ 
forms an orthonormal basis for $\Sc_p(\R^d)$, where $h_n^p:= (2|n|+d)^{-p} 
h_n$.\\
Consider the derivative maps denoted by $\partial_i:\Sc(\R^d)\to
\Sc(\R^d)$ for $i=1,\cdots,d$. We can extend these maps by duality to
$\partial_i:\Sc'(\R^d) \to \Sc'(\R^d)$ as follows: for $\psi \in
\Sc'(\R^d)$,
\[\inpr{\partial_i \psi}{\phi}:=-\inpr{\psi}{\partial_i \phi}, \; \forall \phi
\in \Sc(\R^d).\]
Let $\{e_i: i=1,\cdots,d\}$ be the standard basis vectors in $\R^d$. Then 
for any $n =
(n_1,\cdots,n_d) \in {\mathbb Z}_+^d$ we have (see
\cite{MR562914}*{Appendix A.5})
\[\partial_i h_n =
\sqrt{\frac{n_i}{2}}h_{n-e_i}-\sqrt{\frac{n_i+1}{2}}h_{n+e_i},\]
with the convention that for a multi-index $n = (n_1,\cdots,n_d)$, if $n_i
< 0$ for some $i$, then $h_n \equiv 0$. Above recurrence implies that
$\partial_i:\Sc_{p}(\R^d)\to\Sc_{p-\frac{1}{2}}(\R^d)$ is a bounded linear
operator.\\
For $x \in \R^d$, let $\tau_x$ denote the translation operators on $\Sc(\R^d)$ 
defined by
$(\tau_x\phi)(y):=\phi(y-x), \, \forall y \in \R^d$. This operators can be
extended to $\tau_x:\Sc'(\R^d)\to \Sc'(\R^d)$ by
\[\inpr{\tau_x\phi}{\psi}:=\inpr{\phi}{\tau_{-x}\psi},\, \forall \psi \in
\Sc(\R^d).\]
\begin{propn}\label{tau-x-estmte}
The translation operators $\tau_x, x \in \R^d$ have the following properties:
\begin{enumerate}[label=(\alph*)]
\item For $x \in \R^d$ and any $p \in \R$, $\tau_x: \Sc_p(\R^d)\to\Sc_p(\R^d)$
is a bounded linear map. In particular, there exists a real polynomial $P_k$ of
degree $k = 2(\floor{|p|}+1)$ such that
\[\|\tau_x\phi\|_p\leq P_k(|x|)\|\phi\|_p, \, \forall \phi \in \Sc_p(\R^d).\]
\item For any $x \in \R^d$ and any $i=1,\cdots,d$ we have
\[\tau_x\partial_i = \partial_i\tau_x.\]
\end{enumerate}
\end{propn}
\begin{proof}
See \cite{MR1999259}*{Theorem 2.1} for the proof of part $(a)$. We prove part 
$(b)$.\\
Fix an element $\psi \in \Sc(\R^d)$. Then for $y \in \R^d$,
\[(\tau_{x}\partial_i \psi)(y) = (\partial_i \psi)(y-x)
= \partial_i (\psi(y-x))
= (\partial_i\tau_{x}\psi)(y)\]
i.e. $\tau_{x}\partial_i \psi = \partial_i\tau_{x}\psi$. Via duality we can 
prove $\partial_i \tau_x \phi =  \tau_x \partial_i \phi$ for all $\phi \in
\Sc'(\R^d)$.
\end{proof}
\section{Stochastic Integrals}
In this section, we review basic properties of stochastic integrals, 
specifically those with Hermite-Sobolev valued integrands. Unless stated 
otherwise the we shall use the following notations
throughout this section. Let
$(\Omega, \F, (\F_t)_{t \geq 0},P)$ be a filtered complete probability 
space satisfying the usual conditions. For any real valued
martingale $M$ or a process of finite variation $A$ or a semimartingale $X$, we
assume $M_0 \equiv 0,A_0 \equiv 0,X_0 \equiv 0$. Unless 
stated otherwise stopping times or adapted processes will be 
with respect to the filtration $(\F_t)$ and any such 
real valued process (martingales, processes of finite variation or 
semimartingales) will be assumed to have rcll (right continuous with left 
limits) paths. For our purpose, we do not require the full generality of 
stochastic integration on Hilbert spaces as given in \cite{MR688144}*{Chapter 4 
and 
5}.
\subsection{Stochastic integral with respect to a real valued local $\Ltwo$ 
martingale}
Let $\{M_t\}$ be a real valued $(\F_t)$ adapted local $\Ltwo$ martingale
with rcll paths and $M_0 = 0$. Let $\{\predbracket{M}_t\}$ denote the 
predictable quadratic variation of $M$.
\begin{propn}
Let $p \in \R$. Let $\{G_t\}$ be an 
$\Sc_{-p}(\R^d)$ valued predictable 
process such that 
there exists a localizing sequence $\{\tau_n\}$ with the following property: 
for 
all $t > 0$ and all positive integers $n$,
\[\Exp \int_0^{t\wedge \tau_n} \|G_s\|^2_{-p}\, d\predbracket{M}_s < \infty.\]
Then $\{\int_0^t G_s\, dM_s\}$ is an $\Sc_{-p}(\R^d)$ valued $\{\F_t\}$ 
adapted local $\Ltwo$ martingale.\\
Let 
$\mathbb{K}$ be a 
real separable Hilbert 
space and $T:\Sc_{-p}(\R^d)\to\mathbb{K}$ be a bounded linear operator. Then 
a.s. $t 
\geq 0$,
\[T\int_0^t G_s\, 
dM_s = \int_0^t TG_s\, 
dM_s.\]
In particular, for any $\phi\in \Sc(\R^d)$, a.s. for all $t \geq 0$
\begin{equation}\label{M-st-intg-rlf}
\inpr{\int_0^t G_s\,dM_s}{\phi}=\int_0^t \inpr{G_s}{\phi}\, dM_s.
\end{equation}
\end{propn}

\begin{proof}
For simplicity, assume that $M$ is an $\Ltwo$ martingale and $\tau_n = \infty, 
\forall n$. If $G$ is a predictable step process of the form: $G: = \sum_{i=1}^n 
\indicator{(t_{i-1},t_i]}\,g_i$ where $n$ is a 
positive 
integer, $t_0,t_1,\cdots,t_n$ are real numbers satisfying $0 \leq t_0 < t_1 < 
\cdots t_n$ and $g_i$ are an 
$\Sc_{-p}(\R^d)$ valued, $\F_{t_{i-1}}$ measurable random variable. Then
\[\int_0^t G_s \, dM_s := \sum_{i=1}^n 
(M_{t\wedge t_i}-M_{t\wedge t_{i-1}})g_i\]
and
\begin{align*}
\Exp \left\| \int_0^t G_s \, dM_s \right\|^2_{-p} 
&= \Exp \sum_{i = 1}^n \|g_i\|^2_{-p}(M_{t\wedge t_i}-M_{t\wedge 
t_{i-1}})^2\\
&= \Exp \sum_{i = 1}^n \|g_i\|^2_{-p}(\predbracket{M}_{t\wedge t_i} - 
\predbracket{M}_{t\wedge 
t_{i-1}})\\
&= \Exp \int_0^t \|G_s\|^2_{-p} \, d\predbracket{M}_s.
\end{align*}
In view of the above isometry we can extend the stochastic integral to 
predictable processes $\{G_t\}$ satisfying the integrability condition as 
mentioned in the statement. Proofs of $(\F_t)$ adaptedness and rcll paths are 
standard.\\
If $T:\Sc_{-p}(\R^d)\to\mathbb{K}$ is a bounded linear operator, then a.s. $t 
\geq 0$,
\[T\int_0^t G_s\, 
dM_s = \int_0^t TG_s\, 
dM_s\]
holds for predictable step processes. The relation then extends to all $G$ 
satisfying the integrability condition.
\end{proof}

\begin{rem}
\begin{enumerate}
\item If the martingale is continuous, then we can define the integrals for 
integrands $\{G_t\}$
which are progressively measurable.
\item For continuous processes, equation \eqref{M-st-intg-rlf} was 
pointed out in \cite{MR1837298}*{Proposition 
1.3(a)}.
\end{enumerate}
\end{rem}

\subsection{Stochastic Integral with respect to a real finite variation process}
Let $\{A_t\}$ be a real valued $(\F_t)$-adapted process of 
finite
variation
with right continuous paths. We denote its total variation process by 
$\{V_{[0,t]}(A_{\cdot})\}$.\\
Let $\{G_t\}$ be an $\Sc_{-p}(\R^d)$ valued norm-bounded (i.e. there
exists a constant $R > 0$ such that a.s. $\|G_t\|_{-p} \leq R$ for all $t$)
predictable process.\\
Observe that for all $t \geq 0$ and all $\omega$
\begin{equation}\label{intg-cond-fv}
\int_0^{t} \|G_s\|_{-p}\, |dA_s| \leq R.V_{[0,t]}(A_{\cdot}) < \infty,
\end{equation}
which allows us to define the $\Sc_{-p}(\R^d)$ valued random variable $\int_0^t 
G_s\, 
dA_s$ as a Bochner integral. Note that the process $\{\int_0^t 
G_s\, 
dA_s\}$ is of 
finite variation and has rcll paths. Furthermore, for each $\phi 
\in \Sc(\R^d)$, $t \geq 0$ we have 
\begin{equation}\label{A-st-intg-rlf}
\inpr{\int_0^t G_s\,dA_s}{\phi}=\int_0^t \inpr{G_s}{\phi}\, dA_s.
\end{equation}
For continuous processes, this result was 
pointed out in \cite{MR1837298}*{Proposition 
1.3(a)}.

\subsection{Stochastic Integral with respect to a real semimartingale}
We recall the decomposition of local martingales (\cite{MR1464694}*{Lemma
23.5}, \cite{MR2020294}*{Chapter III, Theorem 25}).
\begin{thm}[Decomposition of Local Martingales]
Given a local martingale $\{M_t\}$, there exist two local martingales
$\{M_t'\}$, $\{M_t''\}$ one of which has bounded jumps and the other is
of locally integrable variation and a.s.
\[M_t = M_t' + M_t'', \, \forall t \geq 0.\]
\end{thm}
Since any local martingale with bounded jumps is locally $\Ltwo$, any 
real semimartingale $X$ has a decomposition (not necessarily unique), a.s. $X_t 
= 
M_t + A_t, t \geq 0$, where $\{M_t\}$ is a local $\Ltwo$ martingale and 
$\{A_t\}$ is a process of finite variation.
\\
Let $\{G_t\}$ be an $\Sc_{-p}(\R^d)$ valued norm-bounded (i.e. there
exists a constant $R > 0$ such that a.s. $\|G_t\|_{-p} \leq R$ for all $t$)
predictable process. Now
the stochastic integral of $\{G_t\}$ with respect to $\{X_t\}$ is defined to be:
\[\int_0^tG_s\,dX_s:=\int_0^tG_s\,dM_s+\int_0^tG_s\,dA_s,\, t \geq 0.\]
\begin{thm}\label{well-defined-integral}
The process $\{\int_0^tG_s\,dX_s\}$ is well-defined, i.e. the definition does 
not depend on 
the decomposition 
$X=M+A$.
\end{thm}
To prove this, we first recall the following result.
\begin{lem}\label{intg-equiv}
Let $\{V_t\}$ be a real valued bounded predictable processes. Let $\{M_t\}$
be an $\Ltwo$ martingale and $\{A_t\}$ be a process of finite
variation such that a.s.
\[M_t = A_t,\, \forall t \geq 0.\]
Then a.s.
\[\int_0^t V_s \, dM_s=\int_0^t V_s \, dA_s,\, \forall t \geq 0.\]
\end{lem}
\begin{proof}
This result is included in the proof of Theorem 23.4 in \cite{MR1464694}.
\end{proof}
\begin{proof}[Proof of Theorem \ref{well-defined-integral}]
First assume that $M$ is an $\Ltwo$ martingale. By Lemma \ref{intg-equiv}, for 
each 
$\phi \in \Sc(\R^d)$, the 
process 
$\{\int_0^t\inpr{G_s}{\phi}dM_s+\int_0^t\inpr{G_s}{\phi}dA_s\}$ does not 
depend on the decomposition $X=M+A$. Now varying $\phi$ in the 
countable set $\{h_n: n\in \mathbb{Z}^d_+\}$, we get a common null set 
$\widetilde \Omega$ such that for all $\omega \in 
\Omega\setminus\widetilde\Omega$, for all $n \in \mathbb{Z}^d_+$ and for all $t 
\geq 0$, we have
\[\inpr{\int_0^t G_s\,dM_s+\int_0^t G_s\,dA_s}{h_n}=\int_0^t\inpr{
G_s }{h_n}\,dM_s+\int_0^t\inpr{G_s}{h_n}\,dA_s.\]
This identifies the $\Sc_{-p}(\R^d)$ process 
$\{\int_0^t G_s\,dM_s+\int_0^t G_s\,dA_s\}$ independent of the decomposition 
$X=M+A$.\\
If $M$ is a local $\Ltwo$ martingale, the proof can be completed using stopping 
time arguments.
\end{proof}
\section{The It\={o} Formula}
Given $\phi \in \Sc'(\R^d)$, there exists a $p > 0$ such that $\phi 
\in \Sc_{-p}(\R^d)$. Let $X_t=(X^1_t,\cdots,X^d_t)$ be an $\R^d$ valued 
$(\F_t)$ semimartingale with rcll paths with the decomposition a.s.
\[X_t = X_0+M_t+A_t,\, t \geq 0\]
where 
$M_t=(M^1_t,\cdots,M^d_t)$ is an $\R^d$ valued locally square integrable 
martingale and $A_t=(A^1_t,\cdots,A^d_t)$ is an $\R^d$ valued process of finite 
variation. Both $\{M_t\}$ and $\{A_t\}$ have rcll paths and 
$M_0 = 0 = A_0$ a.s. By Proposition \ref{tau-x-estmte}, $\{\tau_{X_t}\phi\}$ is
an $\Sc_{-p}(\R^d)$ valued process. Recall that the process $\{X_{t-}\}$ 
defined by
\[X_{t-} := \begin{cases}
X_0,\, \text{if}\, t = 0.\\
\lim_{s \downarrow t}X_s,\, \text{if}\, t > 0.
\end{cases},\]
is predictable (see \cite{MR1943877}*{Chapter I, 2.6 Proposition}).
\begin{lem}\label{tau-predictable-processes}
Let $\phi, \{X_t\}$ be as above. Then for any $1 \leq i \leq d$ and $1 \leq j 
\leq d$,
\begin{enumerate}
\item $\{\tau_{X_{t-}}\phi\}$ is an $\Sc_{-p}(\R^d)$ valued predictable process.
\item $\{\partial_i\tau_{X_{t-}}\phi\}$ is an $\Sc_{- p - 
\frac{1}{2}}(\R^d)$ valued predictable process.
\item $\{\partial_{ij}^2\tau_{X_{t-}}\phi\}$ is an $\Sc_{- p - 1}(\R^d)$ 
valued predictable process.
\end{enumerate}
\end{lem}
\begin{proof}
Since $\{X_{t-}\}$ is predictable and $x 
\mapsto 
\tau_x\phi:\R^d \to \Sc_{-p}(\R^d)$ is continuous (see the proof of 
\cite{MR2373102}*{Proposition 3.1}), the process $\{\tau_{X_{t-}}\phi\}$ is 
predictable.\\
For any $1 \leq i \leq d$, we have $\tau_x(\partial_i\phi) = 
\partial_i\tau_x\phi$ (see Proposition \ref{tau-x-estmte}) and 
$\partial_i:\Sc_{-p}(\R^d)\to\Sc_{-p - 
\frac{1}{2}}(\R^d)$ is a bounded linear operator. Hence 
$\{\partial_i\tau_{X_{t-}}\phi\}$ is an $\Sc_{- p - 
\frac{1}{2}}(\R^d)$ valued predictable process.\\
Similarly for $1 \leq 
i,j \leq d$, the processes $\{\partial_{ij}^2\tau_{X_{t-}}\phi\}$ are  
$\Sc_{- p - 1}(\R^d)$ valued predictable processes.
\end{proof}
Using \cite{MR688144}*{25.5 Corollary 3}, there exists a set 
$\widetilde\Omega$ 
with $P(\widetilde\Omega) = 1$ such that
\[\sum_{s \leq t} |\bigtriangleup 
X_s|^2 < \infty,\, \forall t > 0, \omega 
\in \widetilde\Omega.\]
If $\omega \in \widetilde\Omega$, then there are at most countably many 
jumps of $X$ on $[0,t]$. The following can be easily established.

\begin{lem}\label{jumps-of-X}
Fix $\omega \in \widetilde\Omega$.
\begin{enumerate}[label=(\roman*),ref=\ref{jumps-of-X}(\roman*)]
\item\label{tn-increasing} Fix $t > 0$. Let $\{t_n\}$ be a strictly 
increasing sequence converging to $t$. Then
\[\lim_{n\to \infty} \sum_{s \leq t_n} |\bigtriangleup X_s(\omega)|^2  = 
\sum_{s < 
t} |\bigtriangleup X_s(\omega)|^2.\]
\item\label{tn-decreasing} Fix $t \geq 0$. Let $\{t_n\}$ be a strictly 
decreasing sequence converging to $t$. Then
\[\lim_{n\to \infty} \sum_{t_m < s \leq t_1} |\bigtriangleup X_s(\omega)|^2  = 
\sum_{t 
< s \leq 
t_1} |\bigtriangleup X_s(\omega)|^2.\]
\end{enumerate}
\end{lem}

Using Lemma \ref{jumps-of-X}, we get the following estimate which we use later 
in 
Theorem \ref{Ito-formula}.
\begin{lem}\label{translates-by-jumpsofX}
Let $\phi, \{X_t\}$ be as above. Fix $\omega \in \widetilde\Omega$. Fix $\psi 
\in \Sc(\R^d)$. Then for all $s 
\leq t$ 
\[\left|\inpr{\tau_{X_s}\phi - \tau_{X_{s-}}\phi
+\sum_{i=1}^d (\bigtriangleup 
X^i_s\,\partial_i\tau_{X_{s-}}\phi)}{\psi}\right| \leq C(t,\omega) .\, 
|\!\bigtriangleup X_s|^2 \|\psi\|_{p+1},\]
and hence
\begin{equation}\label{tau-jumpsofX-bnd}
\|\tau_{X_s}\phi - 
\tau_{X_{s-}}\phi
+\sum_{i=1}^d (\bigtriangleup 
X^i_s\,\partial_i\tau_{X_{s-}}\phi)\|_{-p-1} \leq C(t,\omega). 
|\!\bigtriangleup X_s|^2.
\end{equation}
Here $C(t,\omega)$ is a positive constant depending on $t,\omega$ 
and is also non-decreasing in $t$. In particular, \[\tau_{X_t}\phi - 
\tau_{X_{t-}}\phi
+\sum_{i=1}^d (\bigtriangleup 
X^i_t\,\partial_i\tau_{X_{t-}}\phi) = 0,\, \text{if}\; |\!\bigtriangleup X_t| 
= 0.\]
\end{lem}

\begin{note}
To simplify 
notations, we shall write $C(t)$ instead of $C(t,\omega)$.
\end{note}

\begin{proof}[Proof of Lemma \ref{translates-by-jumpsofX}]
By \cite{MR1837298}*{Proposition 1.4}, there exists some positive 
integer $n$ such that the map $x \mapsto \tau_x\phi \in \Sc_{-n}(\R^d)$ is a 
$C^2$ 
map. For any fixed $\psi \in
\Sc(\R^d)$ we have $x \mapsto \inpr{\tau_x \phi}{\psi}$ is a $C^2$ map and
\begin{align*}
\partial_i \inpr{\tau_x \phi}{\psi} &= \partial_i \inpr{\phi}{\psi(\cdot+x)} = 
\inpr{\phi}{\partial_i\psi(\cdot+x)}\\
&= 
\inpr{\phi}{\tau_{-x}\partial_i\psi} = - 
\inpr{\partial_i\tau_x\phi}{\psi}.
\end{align*}
For any $1 \leq i,j \leq d$, we have $\partial_{ij}^2 = \partial_i \partial_j 
= \partial_j \partial_i$ on $\Sc'(\R^d)$ and hence $\partial_{ij}^2: 
\Sc_{-p}(\R^d) \to \Sc_{-p-1}(\R^d)$ is a bounded linear operator. Then there 
exists a constant $\alpha>0$ such
that
\begin{equation}\label{sq-partial-bnd}
\|\partial_{ij}^2 \theta\|_{-p-1} \leq \alpha \|\theta\|_{-p},\,, \forall 
\theta \in \Sc_{-p}(\R^d).
\end{equation}
We follow the proof of \cite{MR1464694}*{Theorem 23.7} and define 
$B(t,\omega):= 
\{x \in \R^d
: |x| \leq \sup_{s \leq t}|X_s(\omega)|\}$. Then using Taylor's formula for the 
$C^2$ 
map $x \mapsto \inpr{\tau_x \phi}{\psi}$, we have for all $s \leq t$
\begin{align*}\label{Ito-formula-id2}
&\left| 
\inpr{\tau_{X_s}\phi - \tau_{X_{s-}}\phi
+\sum_{i=1}^d (\bigtriangleup X^i_s\,\partial_i\tau_{X_{s-}}\phi)}{\psi} 
\right|\\
=&\left|\inpr{\tau_{X_s}\phi}{\psi}-\inpr{\tau_{X_{s-}}\phi}{\psi}
+\sum_{i=1}^d \inpr{\partial_i\tau_{X_{s-}}\phi}{\psi}\,
\bigtriangleup X^i_s\right|\\
=&\left|\inpr{\tau_{X_s}\phi}{\psi}-\inpr{\tau_{X_{s-}}\phi}{\psi}
-\sum_{i=1}^d \partial_i\inpr{\tau_{X_{s-}}\phi}{\psi}\,
\bigtriangleup X^i_s\right|\\
\leq&\frac{1}{2}.|\bigtriangleup X_s|^2\,
\left(\sum_{i,j=1}^d \sup_{y \in
B(t,\omega)}|\inpr{\partial_{ij}^2\tau_y\phi}{\psi}|\right)\\
\leq&\frac{1}{2}.|\bigtriangleup X_s|^2
\,\left(\sum_{i,j=1}^d \sup_{y \in
B(t,\omega)}\|\partial_{ij}^2\tau_y\phi\|_{-p-1}\right)\|\psi\|_{p+1}\\
\leq & \frac{\alpha}{2}.|\bigtriangleup X_s|^2\,
\left(\sup_{y \in 
B(t,\omega)}\|\tau_y\phi\|_{-p}\right)\|\psi\|_{p+1}\,(\text{using}\, 
\eqref{sq-partial-bnd}).
\end{align*}
Define $C(t,\omega) := \frac{\alpha}{2}\left(\sup_{y \in 
B(t,\omega)}\|\tau_y\phi\|_{-p}\right)$. Then $C(t,\omega)$ is non-decreasing 
in $t$ and for 
all $s 
\leq t$ 
\[\left|\inpr{\tau_{X_s}\phi - \tau_{X_{s-}}\phi
+\sum_{i=1}^d (\bigtriangleup 
X^i_s\,\partial_i\tau_{X_{s-}}\phi)}{\psi}\right| \leq C(t,\omega) .\, 
|\!\bigtriangleup X_s|^2 \|\psi\|_{p+1}.\]
From above estimate we have
\[\|\tau_{X_s}\phi - 
\tau_{X_{s-}}\phi
+\sum_{i=1}^d (\bigtriangleup 
X^i_s\,\partial_i\tau_{X_{s-}}\phi)\|_{-p-1} \leq C(t,\omega). 
|\!\bigtriangleup X_s|^2.\]
In particular $\tau_{X_t}\phi - 
\tau_{X_{t-}}\phi
+\sum_{i=1}^d (\bigtriangleup 
X^i_t\,\partial_i\tau_{X_{t-}}\phi) = 0$ if $|\!\bigtriangleup X_t| 
= 0$.
\end{proof}
For any 
$i,j=1,\cdots,d$, 
let $\{[X^i,X^j]_t^c\}$ denote the continuous part 
of $\{[X^i,X^j]_t\}$. We now prove the main result of this paper.
\begin{thm}\label{Ito-formula}
Let $p > 0$ and $\phi \in \Sc_{-p}(\R^d)$. Let $X=(X^1,\cdots,X^d)$ be a $\R^d$
valued $(\F_t)$ semimartingale. Let $\bigtriangleup X^i_s$ denote the jump of 
$X^i_s$. 
Then $\{\tau_{X_t}\phi\}$ is an $\Sc_{-p}(\R^d)$ valued semimartingale and
\[\sum_{s \leq t}\left[\tau_{X_s}\phi - \tau_{X_{s-}}\phi +
\sum_{i=1}^d (\bigtriangleup X^i_s\,\partial_i\tau_{X_{s-}}\phi)\right]\] is a 
$\Sc_{-p-1}(\R^d)$ valued process of finite variation and we have the following 
equality 
in 
$\Sc_{-p-1}(\R^d)$, a.s.
\begin{equation}\label{Ito-formula-statement}
\begin{split}
\tau_{X_t}\phi &= \tau_{X_0}\phi - \sum_{i=1}^d \int_0^t
\partial_i\tau_{X_{s-}}\phi\,
dX^i_s + \frac{1}{2}\sum_{i,j=1}^d \int_0^t \partial_{ij}^2\tau_{X_{s-}}\phi\,
d[X^i,X^j]^c_s\\
&+\sum_{s \leq t}\left[\tau_{X_s}\phi - \tau_{X_{s-}}\phi +
\sum_{i=1}^d (\bigtriangleup X^i_s\,\partial_i\tau_{X_{s-}}\phi)\right], \, t 
\geq 0.
\end{split}
\end{equation}
\end{thm}
\begin{proof}
We proceed in steps.
\begin{enumerate}[label=Step \arabic*:]
\item Let $\widetilde \Omega$ be as in Lemma \ref{translates-by-jumpsofX}. 
Then 
$\omega \in \widetilde\Omega$ implies
(see equation \eqref{tau-jumpsofX-bnd})
\begin{equation}
\sum_{s \leq t} \| \tau_{X_s}\phi - 
\tau_{X_{s-}}\phi
+\sum_{i=1}^d (\bigtriangleup 
X^i_s\,\partial_i\tau_{X_{s-}}\phi)\|_{-p-1} \leq C(t)\sum_{s \leq t} 
|\!\bigtriangleup X_s|^2 < \infty.
\end{equation}
Recall that if $\omega \in \widetilde\Omega$, then there are at most countably 
many 
jumps of $X$ on $[0,t]$. In view of 
the
above estimate we define for any $t \geq 0$
\[Y_t(\omega):=
\sum_{s \leq t} \left[ \tau_{X_s(\omega)}\phi - 
\tau_{X_{s-}(\omega)}\phi
+\sum_{i=1}^d (\bigtriangleup 
X^i_s(\omega)\,\partial_i\tau_{X_{s-}(\omega)}\phi) \right],\, \omega  \in 
\widetilde\Omega
\]
and set $Y_t(\omega):= 0,\, \omega  \in (\widetilde\Omega)^c.$
Then $\{Y_t\}$ is a well-defined $\Sc_{-p-1}(\R^d)$ valued $(\F_t)$ 
adapted process.
\item Now we show $\{Y_t\}$ has rcll paths and is a process of finite 
variation. Fix $\omega \in \widetilde\Omega$. We claim
\begin{enumerate}[label=(\roman*)]
\item $Y_{t-} = \sum_{s < t} \left[ \tau_{X_s}\phi - 
\tau_{X_{s-}}\phi
+\sum_{i=1}^d (\bigtriangleup 
X^i_s\,\partial_i\tau_{X_{s-}}\phi) \right], \, t > 0$.
\item $Y_{t+} = \sum_{s \leq t} \left[ \tau_{X_s}\phi - 
\tau_{X_{s-}}\phi
+\sum_{i=1}^d (\bigtriangleup 
X^i_s\,\partial_i\tau_{X_{s-}}\phi) \right] = Y_t,\, t \geq 0$.
\end{enumerate}
We prove (i). Let $\{t_m\}$ be an increasing sequence converging to 
$t$. Then
\begin{align*}
&\left\|\sum_{s < t} \left[ \tau_{X_s}\phi - 
\tau_{X_{s-}}\phi
+\sum_{i=1}^d (\bigtriangleup 
X^i_s\,\partial_i\tau_{X_{s-}}\phi) \right] - Y_{t_m}\right\|_{-p-1}\\
&= \left\|\sum_{t_m < s < t} \left[ \tau_{X_s}\phi - 
\tau_{X_{s-}}\phi
+\sum_{i=1}^d (\bigtriangleup 
X^i_s\,\partial_i\tau_{X_{s-}}\phi) \right]\right\|_{-p-1}\\
&\leq \sum_{t_m < s < t} \left\| \tau_{X_s}\phi - 
\tau_{X_{s-}}\phi
+\sum_{i=1}^d (\bigtriangleup 
X^i_s\,\partial_i\tau_{X_{s-}}\phi) \right\|_{-p-1}\\
&\leq C(t)\sum_{t_m < s < t} 
|\!\bigtriangleup X_s|^2 \,(\text{using}\, \eqref{tau-jumpsofX-bnd})\\
&= C(t) \left[\sum_{s < 
t} |\bigtriangleup X_s|^2 - \sum_{s \leq t_m} |\bigtriangleup 
X_s|^2   
\right] \xrightarrow{m \to \infty} 0 \,(\text{by Lemma}\, 
\ref{tn-increasing}).
\end{align*}
This proves (i). Proof of (ii) is similar. Now using (i),(ii) we have on 
$\widetilde \Omega$
\[\bigtriangleup Y_t = \tau_{X_t}\phi - 
\tau_{X_{t-}}\phi
+\sum_{i=1}^d (\bigtriangleup 
X^i_t\,\partial_i\tau_{X_{t-}}\phi),\]
and $\bigtriangleup Y_t = 0$ if $\bigtriangleup X_t  = 0$. Now using 
\eqref{tau-jumpsofX-bnd}, we also have
\[\sum_{s \leq t}\|\bigtriangleup 
Y_s\|_{-p-1} \leq C(t)\sum_{s \leq t} 
|\!\bigtriangleup X_s|^2 < \infty,\; \omega \in \widetilde \Omega\]
and $Y_t = \sum_{s \leq t} \bigtriangleup Y_s$. We have shown $\{Y_t\}$ has 
rcll paths. Now we show that $\{Y_t\}$ has paths of finite variation.\\
Let $\omega \in \widetilde\Omega$ and $t > 0$. Let $\Partition = \{0=t_0 < 
t_1 < \cdots < t_m = t\}$ be a partition of $[0,t]$. Then
\begin{align*}
&\sum_{i=1}^m \|Y_{t_i} - Y_{t_{i-1}}\|_{-p-1}\\
&= \sum_{i=1}^m \left\|\sum_{t_{i-1} < s \leq t_i} \left[ \tau_{X_s}\phi - 
\tau_{X_{s-}}\phi
+\sum_{i=1}^d (\bigtriangleup 
X^i_s\,\partial_i\tau_{X_{s-}}\phi) \right]\right\|_{-p-1}\\
&\leq \sum_{i=1}^m \sum_{t_{i-1} < s \leq t_i} \left\| \tau_{X_s}\phi - 
\tau_{X_{s-}}\phi
+\sum_{i=1}^d (\bigtriangleup 
X^i_s\,\partial_i\tau_{X_{s-}}\phi) \right\|_{-p-1}\\
&= \sum_{s \leq t} \left\| \tau_{X_s}\phi - 
\tau_{X_{s-}}\phi
+\sum_{i=1}^d (\bigtriangleup 
X^i_s\,\partial_i\tau_{X_{s-}}\phi) \right\|_{-p-1}\\
&\leq C(t)\sum_{s \leq t} 
|\!\bigtriangleup X_s|^2.
\end{align*}
Since the quantity $C(t)\sum_{s \leq t} 
|\!\bigtriangleup X_s|^2$ is independent of the choice of the partition 
$\Partition$, we have $\{Y_t\}$ is of finite variation 
with
\[Var_{[0,t]}(Y_{\cdot}) \leq C(t)\sum_{s \leq t} 
|\!\bigtriangleup X_s|^2\]
on $\widetilde \Omega$.
\item To complete the proof we need to verify the 
following equality in $\Sc_{-p-1}(\R^d)$, a.s. for all $t \geq 0$
\begin{equation*}
Y_t = 
\tau_{X_t}\phi -\tau_{X_0}\phi + \sum_{i=1}^d \int_0^t
\partial_i\tau_{X_{s-}}\phi\,
dX^i_s
- \frac{1}{2}\sum_{i,j=1}^d \int_0^t \partial_{ij}^2\tau_{X_{s-}}\phi\,
d[X^i,X^j]^c_s.
\end{equation*}
First we assume that the processes  
$\{X_{t-}\},\{[X^i,X^j]^c_t\},\,i,j=1,\cdots,d$ are 
bounded. Since $\partial_i:\Sc_{-p}(\R^d)\to\Sc_{-p-\frac{1}{2}}(\R^d)$ is a 
bounded linear
operator, by Proposition \ref{tau-x-estmte}, we have for 
all $t \geq 0, 
i=1,\cdots,d$
\[\|\partial_i\tau_{X_{t-}}\phi\|_{-p-\frac{1}{2}}\leq C.
\|\tau_{X_{t-}}\phi\|_{-p}\leq C.P_k(|X_{t-}|)\|\phi\|_{-p}\leq
C',\]
where $C,C' > 0$ are appropriate constants. Similarly, there exists a constant
$C'' > 0$ such that
\[\|\partial_{ij}\tau_{X_{t-}}\phi\|_{-p-1}\leq C'',\, \forall t \geq 0, 
i,j=1,\cdots,d.\]
Hence 
$\{\tau_{X_{t-}}\phi\},\{\partial_i\tau_{X_{t-}}\phi\},\{
\partial^2_{ij}\tau_{X_{t-}}\phi\}$ are norm-bounded predictable processes 
(see Lemma \ref{tau-predictable-processes}). As per the results mentioned in 
the 
previous section, we can define stochastic integrals \[I^1_t:=\sum_{i=1}^d
\int_0^t \partial_i\tau_{X_{s-}}\phi\,
dX^i_s,\quad I^2_t:=\sum_{i,j=1}^d
\int_0^t \partial_{ij}^2\tau_{X_{s-}}\phi\,
d[X^i,X^j]^c_s,\, t \geq 0\]
which are respectively $\Sc_{-p-\frac{1}{2}}(\R^d)$ and
$\Sc_{-p-1}(\R^d)$ valued and have rcll paths.

For $n\in \mathbb{Z}^d_+$ applying the It\={o} formula (see 
\cite{MR1464694}*{Theorem 
23.7}) to
the $C^2$ map $x \mapsto \inpr{\tau_x \phi}{h_n}$ we have, a.s. for all $t \geq 
0$
\begin{align}\label{Ito-psi}
\inpr{\tau_{X_t}\phi}{h_n} &= \inpr{\tau_{X_0}\phi}{h_n} -
\underbrace{\sum_{i=1}^d
\int_0^t \inpr{\partial_i\tau_{X_{s-}}\phi}{h_n}\,
dX^i_s}_{=\inpr{I^1_t}{h_n}}\notag\\
&+\frac{1}{2}\underbrace{\sum_{i,j=1}^d
\int_0^t \inpr{\partial_{ij}^2\tau_{X_{s-}}\phi}{h_n}\,
d[X^i,X^j]^c_s}_{=\inpr{I^2_t}{h_n}}\\
&+\sum_{s\leq t}\left[\inpr{\tau_{X_s}
\phi}{h_n}  - \inpr{\tau_{X_{s-}}\phi}{h_n}+\sum_{i=1}^d
\inpr{\partial_i\tau_{X_{s-}}\phi}{h_n}\,
\bigtriangleup\!X^i_s\right],\notag
\end{align}
where $\bigtriangleup X^i_s$ denotes the jump of $X^i_s$. Now varying $n$ in 
the 
countable set $\mathbb{Z}^d_+$, we get a common null set 
$\widetilde \Omega$ such that for all $\omega \in 
\Omega\setminus\widetilde\Omega$, for all $n \in \mathbb{Z}^d_+$ and for all $t 
\geq 0$, we have
\begin{equation*}
\begin{split}
\left\langle\right.(\tau_{X_t}\phi -\tau_{X_0}\phi &+ \sum_{i=1}^d 
\int_0^t
\partial_i\tau_{X_{s-}}\phi\,
dX^i_s\\
&-\frac{1}{2}\sum_{i,j=1}^d \int_0^t \partial_{ij}^2\tau_{X_{s-}}\phi\,
d[X^i,X^j]^c_s
-Y_t),h_n\left.\right\rangle=0.
\end{split}
\end{equation*}
Recall that $\{h_n^q : n\in \mathbb{Z}_+^d\}$ is an orthonormal basis for 
$\Sc_{-q}(\R^d)$, where $h_n^q = (2k+d)^{-q}$ with $k = |n| = n_1 + \cdots + 
n_d$. From the previous relation, we get the required equality in 
$\Sc_{-p-1}(\R^d)$ for semimartingales 
$\{X_t\}$ such that $\{X_{t-}\},\{[X^i,X^j]^c_t\},\,i,j=1,\cdots,d$ are 
bounded.
\item Now suppose at least one of 
$\{X_{t-}\},\{[X^i,X^j]^c_t\},\,i,j=1,\cdots,d$ is not bounded. Then define 
\[\bar\sigma_n:=\inf\{t \geq 0: |[X^i,X^j]_t^c| 
\geq 
n,\,i,j=1,\cdots,d\}\]
and
\[\widetilde\sigma_n:=\inf\{t \geq 0: |X_t| \geq 
n\},\] 
where $|\cdot|$ represents the Euclidean norms in the appropriate space $\R^m$ 
($m=1$ or $d$). Set $\sigma_n = 
\bar\sigma_n\wedge\widetilde\sigma_n$. Then 
$\{\left([X^i,X^j]^c\right)^{\sigma_n}_t\},\,i,j=1,\cdots,d$ are 
bounded.\\
If $|X_0(\omega)| > n$ for some 
$w$, then $\tau_n(\omega)=0$. Such $\omega$ does not contribute to 
$\sum_{i=1}^d \int_0^{t\wedge\sigma_n}
\|\partial_i\tau_{X_{s-}}\xi\|_{p-\frac{1}{2}}^2\,
d\predbracket{M^i}_s$ etc. So we may assume the processes 
$\{X^{\sigma_n}_{t-}\}$ are bounded. Hence a.s. in $\Sc_{-p-1}(\R^d)$ we have 
for all $t \geq 0$
\begin{align*}
\tau_{X_{t\wedge\sigma_n}}\phi =&\tau_{X_0}\phi + \sum_{i=1}^d 
\int_0^{t\wedge\sigma_n}
\partial_i\tau_{X_{s-}}\phi\,
dX^i_s\\
& - \frac{1}{2}\sum_{i,j=1}^d \int_0^{t\wedge\sigma_n} 
\partial_{ij}^2\tau_{X_{s-}}\phi\,
d[X^i,X^j]^c_s
-Y_t.
\end{align*}
Letting $n$ go to infinity we get the result.
\end{enumerate}
\end{proof}

Given a real valued semimartingale $\{X_t\}$, consider the local
time process denoted by $\{L_t(x)\}_{t\in [0,\infty), x \in \R}$. Note that 
this process is jointly measurable in $(x,t,\omega)$ and for each $x \in 
\R$, $\{L_t(x)\}$ is a continuous adapted process. By the occupation 
density formula \cite{MR2020294}*{p. 216, Corollary 
1}, we have for any $\phi \in 
\Sc$, a.s.
\begin{equation}\label{occupation-density}
\int_{-\infty}^{\infty}L_t(x)\phi(x)\, dx = \int_0^t 
\phi(X_{s-})d\bracket{X}^c_s,
\end{equation}
where $\bracket{X}$ stands for $\bracket{X,X}$ and $\bracket{X}^c$ denotes the 
continuous part of $\bracket{X}$ (also see \cite{MR1467435}*{Proposition 4}). 
By 
\cite{MR2020294}*{p. 216, Corollary 
2} a.s.
\[\int_{-\infty}^{\infty}L_t(x)\, dx = \int_0^t 
d\bracket{X}^c_s,
\]
which shows a.s. for all $t$, the map $x \mapsto L_t(x)$ is integrable. We now 
identify the local time 
process in $\Sc'$.
\begin{propn}\label{local-time}
The $\Sc'$ valued process $\{\int_0^t 
\delta_{X_{s-}}d\bracket{X}^c_s\}$ is $\Sc_{-p}$ 
valued for any $p > \frac{1}{4}$ and for each $t$, $\int_0^t 
\delta_{X_{s-}}d\bracket{X}^c_s$ is given by the 
integrable function $x \mapsto 
L_t(x)$.
\end{propn}
\begin{proof}
Note that for any fixed $x \in \R$, the
distribution $\delta_x$ is in $\Sc_{-p}$ for any $p > \frac{1}{4}$ and 
furthermore for such a $p$ we have $\sup_{x \in \R^d} \|\delta_x\|_{-p} < 
\infty$ (see \cite{MR2373102}*{Theorem 4.1}). Also $\tau_x \delta_0 
= \delta_x, \forall x \in \R^d$. Hence $\{\delta_{X_{t-}}\}$ is an 
$\Sc_{-p}$ valued norm-bounded predictable process (see 
Lemma \ref{tau-predictable-processes}). Then we can define the $\Sc_{-p}$ 
valued 
process $\{\int_0^t 
\delta_{X_{s-}}d\bracket{X}^c_s\}$ for any $p > 
\frac{1}{4}$. But for any integer $n \geq 0$, by \eqref{occupation-density} 
a.s. for all $t 
\geq 0$
\begin{align*}
\inpr{\int_0^t 
\delta_{X_{s-}}\,d\bracket{X}^c_s}{h_n}&= \int_0^t 
\inpr{\delta_{X_{s-}}}{h_n}\,d\bracket{X}^c_s\\
&= \int_0^t 
h_n(X_{s-})\,d\bracket{X}^c_s = \int_{-\infty}^{\infty}L_t(x)h_n(x)\, dx
\end{align*}
Then there exists a $P$ null set $\widetilde\Omega$ such that on 
$\Omega\setminus\widetilde\Omega$ for all integers $n \geq 0$ and all $t \geq 0$
\[\inpr{\int_0^t 
\delta_{X_{s-}}\,d\bracket{X}^c_s}{h_n} = \int_{-\infty}^{\infty}L_t(x)h_n(x)\, 
dx.\]
Since $\{(2n+1)^p h_n : n=0,1,\cdots\}$ is an orthonormal basis for 
$\Sc_{-p}$, for each $t$, the $\Sc'$ valued 
random variable $\int_0^t 
\delta_{X_{s-}}d\bracket{X}^c_s$ is given 
by the function $x \mapsto L_t(x)$.
\end{proof}

We now apply Theorem \ref{Ito-formula} to a L\'{e}vy process to obtain the 
existence of solutions of certain classes of stochastic differential equations 
in the Hermite-Sobolev spaces. This is similar in spirit to the same obtained 
in \cite{MR3063763}*{Theorem 3.4 and Lemma 3.6} for continuous processes.\\
Let $p \in \R$. Let $\phi \in \mathcal{S}_{p}$ 
and $\sigma, b \in \Sc_{-p}$. Let $F, G:\mathcal{S}_{p}\times \R \to 
\R$ and let $\bar F, \bar G: \R\times \R \to \R$ be given by $\bar F(x, \tilde 
x) := F(\tau_x\phi, \tilde x), \quad \bar G(x, \tilde 
x) := G(\tau_x\phi, \tilde x)$. Let $\{B_t\}$ be the standard $(\F_t)$ Brownian 
motion and let $N$ be a Poisson process driven by a L\'evy measure $\nu$. Let 
$\widetilde N$ denote the compensated measure. Assume that $B$ and $N$ 
are 
independent. Let the one-dimensional process $\{X_t\}$ 
satisfy the following 
equation: a.s. $t \geq 0$
\begin{equation}\label{fd-sde}
\begin{split}
X_t &=  \int_0^t \bar b(X_{s-})\, ds + \int_0^t \bar\sigma(X_{s-})\, dB_s\\
&+\int_0^t \int_{(0 < |x| < 1)}  \bar F(X_{s-},x)\, \widetilde
N(dsdx)\\
&+ \int_0^t \int_{(|x| \geq  1)} \bar G(X_{s-},x) \,
N(dsdx),
\end{split}
\end{equation}
where
\begin{enumerate}
\item $\bar\sigma(x) := \inpr{\sigma}{\tau_x\phi}, \bar b(x) := 
\inpr{b}{\tau_x\phi}$ are Lipschitz continuous functions,
\item the coefficients $\bar F, \bar G$ satisfy conditions of 
\cite{MR2512800}*{Chapter 6, Section 2} with $c = 1$. This parameter $c$ 
separates the small and large jumps. We assume the integrability 
condition: a.s.
\[\int_0^t\int_{(0 < |x| < 1)}|\bar 
F(X_{s-},x)|^2\, \nu(dx)ds < \infty,\, \forall t \geq 0.\]
\end{enumerate}
As an application of Theorem \ref{Ito-formula} we get the next result.
\begin{thm}\label{Levy}
The $\Sc_p$ valued process $Y$ defined by $Y_t := \tau_{X_t}\phi$ 
solves the following stochastic differential equation with equality in 
$\Sc_{p-1}$:
\begin{equation}\label{Levy-SPDE}
\begin{split}
Y_t(\phi )
&=\phi + \int_0^t 
A(Y_{s-}(\phi))\, dB_s + \int_0^t L(Y_{s-}(\phi))\, ds\\
&+\int_0^t \int_{(0 < |x| < 1)} \left( \tau_{F(Y_{s-}(\phi),x)} 
-Id + 
F(Y_{s-}(\phi),x)\, \partial\right) Y_{s-}(\phi)\,
\nu(dx)\,ds\\
&+\int_0^t \int_{(0 < |x| < 1)} \left(\tau_{F(Y_{s-}(\phi),x)} 
-Id\right) \,
Y_{s-}(\phi)\,\widetilde N(dsdx)\\
&+ \int_0^t \int_{(|x| \geq  1)} \left(\tau_{G(Y_{s-}(\phi),x)} 
-Id\right) \,Y_{s-}(\phi)\,
N(dsdx),
\end{split}
\end{equation}
where the operators $A, L$ on $\Sc_{p}$ are as follows:
\begin{equation*}
A \phi := -  \inpr{\sigma}{\phi}\, \partial \phi,
\end{equation*}
and
\begin{equation*}
L\phi := \frac{1}{2}  
\inpr{\sigma}{\phi}^2\, \partial^2 \phi - 
\inpr{b}{\phi}\, \partial \phi.
\end{equation*}

\end{thm}

\begin{proof}

Observe that
\begin{equation}\label{jump-form}
\bigtriangleup X_t = \bar F(X_{t-},\bigtriangleup 
X_t)\indicator{(0 < |\bigtriangleup X_t| < 1)} + \bar G(X_{t-},\bigtriangleup 
X_t)\indicator{(|\bigtriangleup X_t| \geq 1)}.
\end{equation}
From \eqref{jump-form} we make two observations. Firstly, $|\bar 
F(X_{t-},\bigtriangleup 
X_t)| \indicator{(0 < |\bigtriangleup X_t| < 1)} \leq 1$. In particular, this 
implies
\[|\bar F(X_{t-},\bigtriangleup 
X_t)|^4 \indicator{(0 < |\bigtriangleup X_t| < 1)} \leq |\bar 
F(X_{t-},\bigtriangleup 
X_t)|^2 \indicator{(0 < |\bigtriangleup X_t| < 1)}.\]
Secondly, we have the following simplification.
\begin{align*}
&\tau_{X_s}\phi - \tau_{X_{s-}}\phi
+\bigtriangleup 
X_s\,\partial\tau_{X_{s-}}\phi\\
&= \left( \tau_{\bigtriangleup X_s} - 
Id\right)\tau_{X_{s-}}\phi
+\bigtriangleup 
X_s\,\partial\tau_{X_{s-}}\phi\\
&=\indicator{(0 < |\bigtriangleup X_s| < 1)}\left( \tau_{\bar 
F(X_{s-},\bigtriangleup 
X_s)} - 
Id + \bar F(X_{s-},\bigtriangleup 
X_s)\, \partial\right)\tau_{X_{s-}}\phi\\
&+ \indicator{(|\bigtriangleup X_s| \geq 1)}\left( \tau_{\bar 
G(X_{s-},\bigtriangleup 
X_s)} - 
Id\right)\tau_{X_{s-}}\phi + \indicator{(|\bigtriangleup X_s| \geq 1)}\,\bar 
G(X_{s-},\bigtriangleup 
X_s)\, \partial\tau_{X_{s-}}\phi.
\end{align*}
Using equation \eqref{tau-jumpsofX-bnd}, we have
\begin{align*}
&\indicator{(0 < |\bigtriangleup X_s| < 1)}\left\|\left( 
\tau_{\bar 
F(X_{s-},\bigtriangleup X_s)} - 
Id + \bar F(X_{s-},\bigtriangleup X_s)\, \partial\right) \tau_{X_{s-}}\phi 
\right\|_{-p-1}\\
&\leq C(s).\indicator{(0 < |\bigtriangleup X_s| < 1)}\,|\bar 
F(X_{s-},\bigtriangleup X_s)|^2,
\end{align*}
where $t \mapsto C(t)$ is a positive non-decreasing function. Then
\begin{align*}
&\int_0^t\int_{(0 < |x| < 1)}\left\|\left( 
\tau_{\bar 
F(X_{s-},x)} - 
Id + \bar F(X_{s-},x)\, \partial\right) \tau_{X_{s-}}\phi 
\right\|_{-p-1}^2 \, \nu(dx)ds\\
&\leq \int_0^t C(s)^2\int_{(0 < |x| < 1)}|\bar 
F(X_{s-},x)|^4\, \nu(dx)ds\\
&\leq C(t)^2 \int_0^t\int_{(0 < |x| < 1)}|\bar 
F(X_{s-},x)|^2\, \nu(dx)ds < \infty
\end{align*}
Similarly
\begin{align*}
&\int_0^t\int_{(0 < |x| < 1)}\left\|\left( 
\tau_{\bar 
F(X_{s-},x)} - 
Id \right) \tau_{X_{s-}}\phi 
\right\|_{-p-\frac{1}{2}}^2 \nu(dx)ds\\ 
&\leq \tilde C(t)^2 \int_0^t \int_{(0 < |x| < 1)}|\bar 
F(X_{s-},x)|^2 \nu(dx)ds < \infty,
\end{align*}
where $t \mapsto \tilde C(t)$ is some non-decreasing function. Hence
\begin{align*}
&\sum_{s \leq t} \left[ \tau_{X_s}\phi - \tau_{X_{s-}}\phi
+\bigtriangleup 
X_s\,\partial\tau_{X_{s-}}\phi \right]\\
&= \int_0^t \int_{(0 < |x| < 1)} \left( 
\tau_{\bar 
F(X_{s-},x)} - 
Id + \bar 
F(X_{s-},x)\, \partial\right)\tau_{X_{s-}}\phi\, N(dsdx)\\
&+ \int_0^t \int_{(|x| \geq  1)} \left( 
\tau_{\bar 
G(X_{s-},x)} - 
Id\right)\tau_{X_{s-}}\phi\,
N(dsdx)\\
&+ \int_0^t \int_{(|x| \geq  1)} \bar G(X_{s-},x)\, \partial\tau_{X_{s-}}\phi\,
N(dsdx)\\
&= \int_0^t \int_{(0 < |x| < 1)} \left( 
\tau_{\bar 
F(X_{s-},x)} - 
Id + \bar 
F(X_{s-},x)\, \partial\right)\tau_{X_{s-}}\phi\, \widetilde N(dsdx)\\
&+ \int_0^t \int_{(0 < |x| < 1)} \left( 
\tau_{\bar 
F(X_{s-},x)} - 
Id + \bar 
F(X_{s-},x)\, \partial\right)\tau_{X_{s-}}\phi\, \nu(dx)ds\\
&+ \int_0^t \int_{(|x| \geq  1)} \left( 
\tau_{\bar 
G(X_{s-},x)} - 
Id\right)\tau_{X_{s-}}\phi\,
N(dsdx)\\
&+ \int_0^t \int_{(|x| \geq  1)} \bar G(X_{s-},x)\, \partial\tau_{X_{s-}}\phi\,
N(dsdx).
\end{align*}
Now by the It\={o} formula (Theorem \ref{Ito-formula})
\begin{align*}
\tau_{X_t}\phi &= \tau_{X_0}\phi + \int_0^t 
A(\tau_{X_{s-}}\phi)\, dB_s + \int_0^t L(\tau_{X_{s-}}\phi)\, ds\\
&-\int_0^t \int_{(0 < |x| < 1)} \bar F(X_{s-},x)\, \partial\tau_{X_{s-}}\phi\, 
\widetilde
N(dsdx)\\
&- \int_0^t \int_{(|x| \geq  1)} \bar G(X_{s-},x)\, \partial\tau_{X_{s-}}\phi\,
N(dsdx)\\
&+ \sum_{s \leq t} \left[ \tau_{X_s}\phi - \tau_{X_{s-}}\phi
+\bigtriangleup 
X_s\,\partial\tau_{X_{s-}}\phi \right]\\
&=\phi + \int_0^t 
A(\tau_{X_{s-}}\phi)\, dB_s + \int_0^t L(\tau_{X_{s-}}\phi)\, ds\\
&+\int_0^t \int_{(0 < |x| < 1)} \left( 
\tau_{\bar 
F(X_{s-},x)} - 
Id + \bar F(X_{s-},x)\, \partial\right) \tau_{X_{s-}}\phi\,
\nu(dx)\,ds\\
&+\int_0^t \int_{(0 < |x| < 1)} \left( 
\tau_{\bar F(X_{s-},x)} - 
Id\right)\tau_{X_{s-}}\phi\,\widetilde N(dsdx)\\
&+ \int_0^t \int_{(|x| \geq  1)} \left( 
\tau_{\bar 
G(X_{s-},x)} - 
Id\right)\tau_{X_{s-}}\phi\,
N(dsdx)
\end{align*}
Hence $Y_t(\phi):= \tau_{X_t}\phi$ solves the equation \eqref{Levy-SPDE}.
\end{proof}

\noindent\textbf{Acknowledgement:} The author would like to thank Professor B.
Rajeev, Indian Statistical Institute, Bangalore for valuable suggestions during 
the work and pointing out the way to Theorem \ref{Ito-formula}.


\end{document}